\documentclass[12pt]{amsart}
\usepackage{geometry}
\usepackage{graphicx}
\usepackage{amssymb}
\usepackage{epstopdf}
\usepackage{amsmath,amscd}
\usepackage{amsthm}
\usepackage{url,verbatim}

\RequirePackage[colorlinks,citecolor=blue,urlcolor=blue]{hyperref}
\usepackage{breakurl}
\theoremstyle{plain}
\newtheorem{theorem}{Theorem}
\newtheorem{definition}[theorem]{Definition}
\newtheorem{lemma}[theorem]{Lemma}
\newtheorem{proposition}[theorem]{Proposition}

\newtheorem{example}[theorem]{Example}

\newtheorem{remark}[theorem]{Remark}

\newcommand\ol{\overline}
\newcommand\Aut{\mathrm{Aut}}
\newcommand \sA{\mathcal{A}}

\newcommand\sH{{\mathcal H}}
\newcommand\EE{{\mathbb E}}

\newcommand\RR{{\mathbb R}}
\newcommand\ZZ{{\mathbb Z}}
\newcommand\NN{{\mathbb N}}
\newcommand\QQ{{\mathbb Q}}
\newcommand\PP{{\mathbb P}}

\renewcommand\a{\alpha}
\newcommand\om{\omega}
\newcommand\g{\gamma}
\newcommand\be{\beta}
\newcommand\si{\sigma}
\newcommand\eps{\epsilon}
\newcommand\De{\Delta}
\newcommand\qq{\qquad}
\newcommand\q{\quad}
\newcommand\resp{respectively}

\newcommand\rank{\mathrm{rank}}
\newcommand\oo{\infty}
\newcommand\sG{{\mathcal G}}

\newcommand\sN{{\mathcal N}}
\newcommand\sS{{\mathcal S}}
\newcommand\sW{{\mathcal W}}
\newcommand\Ga{\Gamma}
\newcommand\Si{\Sigma}
\newcommand\de{\delta}

\newcommand\id{{\bf 1}}

\newcommand\dih{\mathrm{Dih}}
\newcommand\bzero{{\bf 0}}

\newcommand\Eq{{\mathrm{Eq}}}

\renewcommand\ell{l}
\newcommand\SL{\text{\rm SL}}

\newcommand\vG{\vec G}

\newcommand\vE{\vec E}
\newcommand\pd{\partial}

\newcommand\ghf{graph height function}
\newcommand\ughf{unimodular graph height function}
\newcommand\GHF{group height function}
\newcommand\Stab{\mathrm{Stab}}
\newcommand\hdi{$\sH$-difference-invariant}
\newcommand\sap{point of increase}
\newcommand\saps{points of increase}

\renewcommand\o{{\mathrm o}}
\newcommand\normal{\trianglelefteq}
\newcommand\normalgr{\trianglerighteq}

\renewcommand\O{\mathrm{O}}

\newcounter{mycount}
\newenvironment{romlist}{\begin{list}{\rm(\roman{mycount})}%
   {\usecounter{mycount}\labelwidth=1cm\itemsep 0pt}}{\end{list}}
\newenvironment{numlist}{\begin{list}{\arabic{mycount}.}%
   {\usecounter{mycount}\labelwidth=1cm\itemsep 0pt}}{\end{list}}
\newenvironment{letlist}{\begin{list}{\rm(\alph{mycount})}%
   {\usecounter{mycount}\labelwidth=1cm\itemsep 0pt}}{\end{list}}
\newenvironment{Alist}{\begin{list}{\MakeUppercase{\alph{mycount}}.}%
   {\usecounter{mycount}\labelwidth=1cm\itemsep 0pt}}{\end{list}}

\numberwithin{equation}{section}
\numberwithin{theorem}{section}
\numberwithin{figure}{section}

\title[Connective constants and height functions]{Connective constants and height functions\\ for Cayley graphs}
\author{Geoffrey R.\ Grimmett}
\address{Statistical Laboratory, Centre for
Mathematical Sciences, Cambridge University, Wilberforce Road,
Cambridge CB3 0WB, UK}
\email{g.r.grimmett@statslab.cam.ac.uk}
\urladdr{\url{http://www.statslab.cam.ac.uk/~grg/}}
\author{Zhongyang Li} 
\address{Department of Mathematics,
University of Connecticut,
Storrs, Connecticut 06269-3009, USA} 
\email{zhongyang.li@uconn.edu}
\urladdr{\url{http://www.math.uconn.edu/~zhongyang/}}

\begin{document}

\begin{abstract}
The connective constant $\mu(G)$ of an infinite transitive graph $G$
is the exponential growth rate of the number of self-avoiding walks 
from a given origin. In earlier work of Grimmett and Li, a locality theorem
was proved for connective constants, namely, that the connective constants
of two graphs are close in value whenever the graphs agree on a large ball around the origin. 
A condition of the theorem was that the graphs support so-called \lq unimodular graph height functions\rq.
When the graphs are Cayley graphs of infinite, finitely generated groups, there is a special
type of \ughf\ termed here a \lq\emph{group} height function\rq. 
A necessary and sufficient condition for the existence of a \GHF\ is presented,
and may be applied in the context of the bridge constant, and of the locality of connective constants for Cayley graphs.
Locality may thereby be established for a variety of infinite groups
including those with strictly positive deficiency.

It is proved that a large class of Cayley graphs
support  \ughf s, that are in addition
\emph{harmonic} on the graph.  
This implies, for example, the existence of \ughf s for the Cayley graphs of finitely generated
solvable groups. It turns out that graphs with non-unimodular automorphism
subgroups also possess
\ghf s, but the resulting \ghf s need not be  harmonic.

Group height functions, as well as the graph height functions of the previous
paragraph, are non-constant harmonic functions with linear growth and an additional
property of having periodic differences. 
The existence of such functions on Cayley graphs is a topic of interest beyond
their applications in the theory of self-avoiding walks.
\end{abstract}

\date{27 January 2015, revised 22 August 2016}

\maketitle

\section{Introduction, and summary of results}

The main purpose of this article
is to study aspects of `locality' for the connective constants of
Cayley graphs of finitely presented groups. The locality
question may be posed as follows: if
two Cayley graphs are locally isomorphic in the
sense that they agree on a large ball centred at the identity,
then are their connective constants close in value? 
The current work may be viewed as a continuation of
the study of locality for connective constants of quasi-transitive graphs
reported in \cite{GL-loc}.
The locality of critical points is a well developed topic
in the theory of disordered systems, and the reader is referred,
for example,  to
\cite{bnp,mt,psn} for related work about percolation on Cayley graphs.
 
The self-avoiding walk (SAW) problem was introduced to
mathematicians in 1954 by Hammersley and Morton \cite{hm}.
Let $G$ be an infinite, connected, transitive graph.
The number of $n$-step SAWs on $G$
from a given origin grows in the manner of
$\mu^{n(1+\o(1))}$ for some growth rate $\mu=\mu(G)$ called the 
\emph{connective constant} of the graph $G$. 
The value of $\mu(G)$ is not generally known,
and a substantial part of
the literature on SAWs is targeted at 
properties of connective constants. The
current paper may be viewed in this light, as a continuation of the
series of papers   \cite{GrLrev, GrL2, GrL1,GL-loc, GrL3}.

The principal result of \cite{GL-loc} is as follows. 
Let $G$, $G'$ be infinite, transitive graphs,
and write $S_K(v,G)$ for the $K$-ball around the vertex $v$ in $G$.
If $S_K(v,G)$ and $S_K(v',G')$
are isomorphic as rooted graphs, then 
\begin{equation}\label{eq:loc1}
|\mu(G)-\mu(G')| \le \eps_K(G),
\end{equation}
where $\eps_K(G)\to 0$ as $K \to\oo$.  
This is proved subject to a 
condition on  $G$ and  $G'$, namely
that they support so-called \lq \ughf s'.

Cayley graphs of finitely generated groups  provide a category
of transitive graphs of special interest. 
They possess an algebraic structure in addition to
their graphical structure, and this algebraic structure provides
a mechanism for the study of their graph height functions. 
A necessary and sufficient condition is
given in Theorem \ref{thm3} for the existence of a so-called
\lq \GHF', and it is pointed out that a \GHF\ is a \ughf, but not \emph{vice versa}.
The class of Cayley groups that possess  \GHF s  
includes all infinite, finitely generated, free solvable groups and free nilpotent groups,
and groups with fewer relators than generators.

We turn briefly to the topic of harmonic functions. 
The study of the existence and structure of non-constant harmonic functions on Cayley graphs has
acquired prominence in geometric group theory through the work of 
Kleiner  and others, see \cite{kle10,st}.  The group height functions
of Section \ref{sec:Cay} are harmonic with linear growth.
Thus, one aspect of the work reported in this paper is the construction, 
on certain classes of Cayley graphs, of 
linear-growth harmonic functions with the additional property of
having differences that are invariant under the action of a subgroup
of automorphisms. Such harmonic functions do not appear to contribute
to the discussion of the Liouville property (see, for example, \cite[Defn 2.1.10] {Kum}),
since both their positive and negative parts are unbounded.  
For recent articles on the related aspect of geometric group theory,
the reader is referred to \cite{MY14,T14}. 

This paper is organized as follows.
Graphs, self-avoiding walks, and Cayley graphs are introduced in Section \ref{sec:not}.
Graph height functions and the locality theorem of \cite{GL-loc} are reviewed
in Section \ref{sec:sec2}, and a principal tool is presented at Theorem \ref{prop:indep}. 
Graphs with non-unimodular automorphism subgroups may be handled by similar means
(see Theorem \ref{cor:nonunim}), but the resulting \ghf s need not
be harmonic. 

Group height functions are the subject of Section \ref{sec:Cay}, and a necessary and
sufficient condition is presented in Theorem \ref{thm3} for the existence of a \GHF.
Section \ref{sce} is devoted to existence conditions for height functions, 
leading to existence theorems for virtually solvable groups.
In Section \ref{sec:conv} is presented a theorem
for the convergence of connective constants subject to the addition of further relators.
This parallels the Grimmett--Marstrand theorem \cite{gm} for the critical percolation probabilities of 
slabs of $\ZZ^d$ (see also \cite[Thm 5.2]{GrL3}).
Sections \ref{sec:pf-}--\ref{sec:nonunim} contain the proofs of Theorems 
\ref{prop:indep} and \ref{cor:nonunim}, \resp.

\section{Graphs, self-avoiding walks, and groups}\label{sec:not}

The graphs $G=(V,E)$ considered here are
infinite, connected, and usually simple.
An undirected edge $e$ with endpoints $u$, $v$ is 
written as $e=\langle u,v \rangle$, and if directed from $u$
to $v$ as $[u,v\rangle$.
If $\langle u,v \rangle \in E$, we call $u$ and $v$ \emph{adjacent}
and write $u \sim v$. The set of neighbours of $v \in V$
is denoted $\pd v:=\{u: u \sim v\}$. 

The \emph{degree} $\deg(v)$ of vertex $v$ is the number of edges
incident to $v$, and $G$ is called \emph{locally finite} if
every vertex-degree is finite.  
The \emph{graph-distance} between two vertices $u$, $v$ is the number of edges
in the shortest path from $u$ to $v$, denoted $d_G(u,v)$.

The automorphism group of the graph $G=(V,E)$ is
denoted $\Aut(G)$. A subgroup $\Ga \le \Aut(G)$ is said to 
\emph{act transitively} on $G$
if, for $v,w\in V$, there exists $\g \in \Ga$ with $\g v=w$.
It is said to  \emph{act quasi-transitively} if there is a finite
set $W$ of vertices 
such that, for $v \in V$, there exist
$w \in W$ and $\g \in \Ga$ with $\g v =w$.
The graph is called (\emph{vertex}-)\emph{transitive} 
(\resp, \emph{quasi-transitive}) if $\Aut(G)$ acts transitively
(\resp, quasi-transitively). 
For $\Ga \le \Aut(G)$ and a vertex $v \in V$, the orbit of $v$ under $\Ga$ is written 
$\Ga v$. 

An (\emph{$n$-step}) \emph{walk} $w$ on $G$ is
an alternating sequence $(w_0,e_0,w_1,e_1,\dots, e_{n-1}, w_n)$,
where $n \ge 0$,  of vertices $w_i$
and edges $e_i=\langle w_i, w_{i+1}\rangle$, and its \emph{length} $|w|$  is the number of its edges.
The walk $w$ is called \emph{closed} if $w_0=w_n$. 
A \emph{cycle} is a closed walk $w$ satisfying $n \ge 3$ and 
$w_i\ne w_j$ for $1 \le i < j \le n$. 

An (\emph{$n$-step}) \emph{self-avoiding walk} (SAW) 
on $G$ is  a walk containing $n$ edges
no vertex of which appears more than once.
Let $\Si_n(v)$ be the set of $n$-step SAWs starting at $v$, with
cardinality $\si_n(v):=|\Si_n(v)|$.
\emph{Assume that $G$ is transitive},
and select a vertex of $G$ which we call the \emph{identity} 
or \emph{origin}, denoted $\id=\id_G$, and let
$\si_n=\si_n(\id)$.
It is standard (see \cite{hm,ms}) that
\begin{equation}\label{eq:sisub}
\si_{m+n} \le \si_m\si_n,
\end{equation}
whence, by the subadditive limit theorem,  the \emph{connective constant}
$$
\mu=\mu(G) :=\lim_{n\to\oo} \si_n^{1/n}
$$
exists. See \cite{bdgs,ms} for recent accounts of the theory of SAWs. 

We turn now to finitely generated groups and their Cayley graphs.
Let $\Ga$ be a group with
generator set $S$ satisfying $|S|<\oo$ 
and $\id\notin S$, where
$\id=\id_\Ga$ is the identity element. 
We write $\Ga =\langle S \mid R\rangle$ with $R$ a set of relators,
and we adopt the following convention for the inverses of generators.
For the sake of concreteness, we consider $S$ as a set of symbols, 
and any information concerning inverses is encoded in the relator set; 
it will always be the case that, using this information, we may identify the inverse
of $s \in S$ as another generator $s'\in S$. For example,
the free abelian group of rank $2$ has presentation $\langle x,y,X,Y\mid xX, yY, xyXY\rangle$,
and the infinite dihedral group $\langle s_1,s_2\mid s_1^2,s_2^2\rangle$. 
Such a group is called \emph{finitely generated} (in that $|S|<\oo$), and \emph{finitely presented}
if, in addition,  $|R|<\oo$.

The \emph{Cayley graph} of $\Ga=\langle S \mid R\rangle$ is the simple graph $G=G(S,R)$
with vertex-set $\Ga$, and an (undirected) edge $\langle \g_1,\g_2\rangle$ if and only
if $\g_2=\g_1s$ for some $s\in S$.
Further properties of Cayley graphs are presented as needed in Section \ref{sec:Cay}.
See \cite{bab95} for an account of Cayley graphs, and \cite{LyP} for a short account.
The books \cite{dlH} and \cite{LedW,Rob} are devoted 
to geometric group theory, and general group theory,
\resp.

The set of integers is written $\ZZ$, the natural numbers as
$\NN$, and the rationals as $\QQ$.

\section{Graph height functions}\label{sec:sec2}

We recall from \cite{GL-loc} the definition
of a \ghf\ for a transitive graph, and then we review the locality theorem (the
proof of which may be found in \cite{GL-loc}).
This is followed by Theorem \ref{prop:indep}, which is
one of the main tools of this work.

Let $\sG$ be the set of all infinite, connected, transitive, locally finite, simple graphs,
and let $G=(V,E)\in\sG$. 
Let $\sH$ be a subgroup of $\Aut(G)$.
A function $F:V \to \RR$ is said to be \emph{\hdi} if 
\begin{equation}\label{eq:hdi2}
F(v)-F(w)=F(\g v)-F(\g w), \qq v,w\in V,\ \g\in \sH.
\end{equation}

\begin{definition} \label{def:height}

A \emph{graph height function} on $G$ is a pair $(h,\sH)$, where 
$\sH\le \Aut(G)$ acts quasi-transitively on $G$ and  $h:V \to \ZZ$, such that
\begin{letlist}
\item $h(\id)=0$,
\item $h$ is \hdi,
\item for  $v\in V$,
there exist $u,w \in \pd v$ such that
$h(u) < h(v) < h(w)$.
\end{letlist}
The \ghf\ $(h,\sH)$ is called \emph{unimodular} if 
$\sH$ is unimodular.
\end{definition}

We remind the reader of the definition of the unimodularity
of a subgroup $\sH \le \Aut(G)$.
The ($\sH$-)\emph{stabilizer} 
 $\Stab_v$ ($=\Stab_v^\sH$) of a vertex $v$ is the set of all $\g\in \sH$ for which $\g v = v$.
As shown in \cite{Trof} (see also \cite{BLPS,LyP,SoW}), when viewed as a topological group
with the  topology of pointwise convergence, $\sH$
is unimodular  if and only if 
\begin{equation}\label{g804}
|\Stab_u v| = |\Stab_v u|, \qquad v\in V,\ u \in \sH v.
\end{equation} 
We follow
\cite[Chap.\ 8]{LyP} by \emph{defining} $\sH$ to be \emph{unimodular} (on $G$)  if
\eqref{g804} holds. 

We sometimes omit the reference to $\sH$ and refer to such $h$ as a
\ghf. In Section \ref{sec:Cay} is defined the related concept of a \emph{\GHF}
for the Cayley graph of a finitely presented group. We shall see that every \GHF\ is
a graph height function, but not \emph{vice versa}.

\begin{remark}[Linear growth]
A \ghf\ $(h,\sH)$ on $G$ has linear growth in that
$h(\gamma^n \id) = nh(\g \id)$ for $\g\in\sH$.
\end{remark}

Associated with the \ghf\ $(h,\sH)$ are two integers $d$, $r$  given as follows.
We set
\begin{equation}\label{eq:defd}
d=d(h)=\max\bigl\{|h(u)-h(v)|: u,v\in V,\ u \sim v\bigr\}.
\end{equation}
If $\sH$ acts transitively, we set
$r=0$. Assume $\sH$ does not act transitively, and 
let $r=r(h,\sH)$ be the least integer $r$ such that the following holds.
For $u,v \in V$ in different orbits of $\sH$, there
exists $v' \in \sH v$ 
such that $h(u)<h(v')$, and a SAW $\nu(u,v')$ from $u$
to $v'$, with length $r$ or less, all of whose vertices $x$, other than
its endvertices,  satisfy $h(u)<h(x)< h(v')$. We recall \cite[Prop.\ 3.2]{GL-loc}
where it is proved, \emph{inter alia}, that $r<\oo$.

We state next the locality theorem for transitive graphs.
The \emph{ball} $S_k=S_k(G)$, with centre $\id=\id_G$ and radius $k$, is
the subgraph of $G$ induced by the set of its vertices within
graph-distance $k$ of $\id$. For $G,G'\in \sG$, we write
$S_k(G) \simeq S_k(G')$ if there exists a  graph-isomorphism from $S_k(G)$ to
$S_k(G')$ that maps $\id_G$ to $\id_{G'}$, and we let
$$
K(G,G') = \max\bigl\{k: S_k(G) \simeq S_k(G')\bigr\}, \qq G,G' \in \sG.
$$
For $D\ge 1$ and $R \ge 0$, let $\sG_{D,R}$ be the set of
all $G\in\sG$ which possess a \ughf\  
$h$ satisfying $d(h)\le D$ and $r(h,\sH)\le R$. 

For $G \in \sG$ with a given \ughf\ $(h,\sH)$, there is a subset of SAWs called \emph{bridges}
which are useful in the study of the geometry of SAWs on $G$. 
The SAW $\pi=(\pi_0,\pi_1,\dots,\pi_n)\in \Si_n(v)$ is called a \emph{bridge} if 
\begin{equation}\label{eq:bridge}
h(\pi_0)<h(\pi_i) \leq h(\pi_n), \qq 1\leq i\leq n,
\end{equation}
and the total number of such bridges is denoted $b_n(v)$.
It is easily seen (as in \cite{HW62})  that $b_n:=b_n(\id)$ satisfies
\begin{equation}\label{eq:b-subadditive}
b_{m+n} \ge b_m b_n,
\end{equation}
from which we deduce the existence of the \emph{bridge constant}
\begin{equation}\label{eq:bexists}
\be  = \be(G) = \lim_{n\to\oo} b_n^{1/n}.
\end{equation}
The definition of $\be$ depends on the choice of height function,
but it turns out that, under reasonable conditions, its value does not.

\begin{theorem}[Bridges and locality for transitive graphs, \cite{GL-loc}]\label{thm2}
\mbox{}
\begin{letlist}
\item
If $G \in \sG$ supports a \ughf\ $(h,\sH)$, then
$\be(G)= \mu(G)$.
\item
Let $D \ge 1$, $R \ge 0$,
and let $G \in \sG$ and $G_m \in \sG_{D,R}$ for $m \ge 1$ be such that
$K(G,G_m)  \to \oo$ as $m \to\oo$.
Then $\mu(G_m) \to \mu(G)$.
\end{letlist}
\end{theorem}

The main thrust of the current paper is to identify classes of finitely generated groups
whose Cayley graphs support \ghf s, and one of our main tools is 
the following theorem, of which the proof is given in Section \ref{sec:pf-}.

\begin{theorem}\label{prop:indep}
Let $G=(V,E)\in \sG$.
Suppose there exist
\begin{letlist}
\item  a subgroup $\Ga\le \Aut(G)$ acting
transitively on $V$, 
\item  a normal subgroup $\sH\normal\Ga$ satisfying
$[\Ga:\sH] < \oo$,  which is unimodular,
\item  a function $F:\sH\id \to \ZZ$ that is \hdi\ and non-constant.
\end{letlist}
Then,
\begin{romlist}
\item there exists a unique harmonic, \hdi\ function $\psi$ on $G$ that agrees with $F$ on $\sH\id$.
\item there exists  a harmonic, \hdi\ function $\psi'$ that increases 
everywhere, in that every $v \in V$ has
neighbours $u$, $w$ such that $\psi'(u)<\psi'(v)<\psi'(w)$, 
\item the function $\psi$ of part (i) takes rational values, and 
the $\psi'$ of part (ii) may be taken to be rational also; therefore, there exists a harmonic, 
\ughf\ of the form $(h, \sH)$.
\end{romlist}
\end{theorem}

The first part of condition (c) is to be interpreted as saying that \eqref{eq:hdi2} holds
for $v,w\in \sH\id$ and $\g \in \sH$. 
Since $G$ is transitive, the choice of origin $\id$ is arbitrary, and hence
the orbit $\sH\id$ may be replaced by any orbit of $\sH$.

One application of Theorem \ref{prop:indep}, or more precisely
of its method of proof, is the proof of the existence of \ghf s for graphs
with quasi-transitive, non-unimodular automorphism subgroups.
See Section \ref{sec:nonunim} for the proof of the following.

\begin{theorem}\label{cor:nonunim}
Let $G=(V,E)\in\sG$.
Suppose there exist a subgroup $\Ga\le \Aut(G)$ acting
transitively on $V$, and a normal subgroup $\sH\normal\Ga$ satisfying
$[\Ga:\sH] < \oo$,   such that $\sH$ is  non-unimodular. 
Then $G$ has a \ghf\ $(h, \sH)$, which is not generally harmonic.
\end{theorem}

The proofs of Theorems \ref{prop:indep} and 
\ref{cor:nonunim} are inspired in part by the proofs of 
\cite[Sect.\ 3]{LeeP} where, \emph{inter alia}, it is explained
that some graphs support harmonic maps, taking values in a function space, with a 
property of equivariance in norm. In this paper,  we study \hdi, integer-valued harmonic functions.

\section{Group height functions}\label{sec:Cay}

We consider  Cayley graphs of
finitely generated groups next, and a type of graph height function
called a \lq \GHF'.  
Let $\Ga$ be a finitely generated group with presentation $\langle S\mid R\rangle$,
as in Section \ref{sec:not}. A \GHF\ on a Cayley graph $G$ of $\Ga$ may in fact be defined 
as a function on the group $\Ga$ itself, but it will be convenient that it acts
on the same domain as a \ghf\ (that is, on $G$ rather than on $\Ga$). 
When viewed as a function on the group $\Ga$,
a \GHF\ is essentially a surjective homomorphism to $\ZZ$, and such functions are 
of importance in group theory (see Remark \ref{rem:betti}). 

Each relator $\rho\in R$ is a
word of the form $\rho = t_1t_2\cdots t_r$ with $t_i \in S$ and $r\ge 1$, and we define the vector
$u(\rho) = (u_s(\rho): s\in S)$ by
$$
u_s(\rho) = |\{i: t_i = s\}| , \qq s\in S.
$$
Let $C$ be the $|R| \times |S|$ matrix with row vectors $u(\rho)$, $\rho \in R$,
called the \emph{coefficient matrix} of the presentation $\langle S\mid R \rangle$.
Its null space $\sN(C)$ is the set of column vectors $\g=(\g_s: s\in S)$ such that $C \g=\bzero$.
Since $C$ has integer entries, $\sN(C)$ is non-trivial if and only if it contains
a non-zero vector of integers (that is, an integer vector
other than the zero vector $\bzero$).
If $\g\in\ZZ^S$ is a non-zero element of $\sN(C)$, then $\g$ gives rise to
a function $h:V\to\ZZ$ defined as follows. Any $v \in V$ may be expressed as a word in the
alphabet $S$, which is to say that $v = s_1 s_2 \cdots s_m$ for some $s_i \in S$ and $m\ge 0$. We set
\begin{equation}\label{eq:gheight}
h(v) = \sum_{i=1}^m \g_{s_i}.
\end{equation}
Any function $h$ arising in this way is called a \emph{\GHF} of the presentation (or of the Cayley graph).
We see next that a \GHF\ is well defined by \eqref{eq:gheight},
and is indeed a graph height function in the sense of Definition
\ref{def:height}. A graph height function, even if unimodular, 
need not be a \GHF\ (see, for 
instance, Example (d) following Remark \ref{rem:betti}).

\begin{theorem}\label{thm3}
Let $G$ be the Cayley graph of the finitely generated group
$\Ga=\langle S\mid R\rangle$, with
coefficient matrix $C$.
\begin{letlist}
\item 
Let $\g=(\g_s: s\in S)\in\sN(C)$ satisfy $\g\in\ZZ^S$, $\g \ne \bzero$.
The \GHF\ $h$ given by \eqref{eq:gheight} is well defined, 
and gives rise to a \ughf\ $(h,\Ga)$ on $G$. 
\item
The Cayley graph $G(S,R)$ of the presentation $\langle S\mid R\rangle$ has a \GHF\ 
if and only if $\rank(C) < |S|$.
\item
A \GHF\ is a group invariant in the  sense that,
if $h$ is a \GHF\ of $G$, then it is also a \GHF\ for
the Cayley graph of any other presentation of $\Ga$.
\item
A \GHF\ $h$ of $G$ is harmonic, in that
$$
h(v) = \frac1{\deg(v)}\sum_{u \sim v} h(u), \qq v \in V.
$$ 
\end{letlist}
\end{theorem}

Since the \GHF\ $h$ of \eqref{eq:gheight} is a graph height function, and $\Ga$
acts transitively, 
\begin{equation}\label{def:d}
d(h) = \max\{\g_s: s \in S\},
\end{equation}
in agreement with \eqref{eq:defd}.
In the light of Theorem \ref{thm3}(c), we may speak of a group possessing a \GHF.

\begin{remark}\label{rem:betti}
The quantity $b(\Ga):=|S|-\rank(C)$ is in fact 
an invariant of $\Ga$, and may be
called the \emph{first Betti number} since it 
equals the power of $\ZZ$ in the abelianization $\Ga/[\Ga,\Ga]$
(see, for example, \cite[Chap.\ 8]{LedW}). 
Group height functions are a standard tool of group theorists, since they are
(when the non-zero $\g_s$ are coprime) surjective homomorphisms
from  $\Ga$ to $\ZZ$. This fact is used, for example,  in the proof of \cite[Thm 4.1]{GL-amen},
which asserts that the Cayley graph of an infinite, finitely generated, elementary amenable group
possesses a harmonic (unimodular) \ghf.  
Further details may be found in \cite{Hill94, Hill09}.

Although some of the arguments of the current paper are standard within group theory, 
we prefer to include sufficient details to aid readers from other backgrounds. 
\end{remark}

It follows in particular from Theorem \ref{thm3} that $G$ has a \GHF\ if $|R| < |S|$,
which is to say that the presentation $\Ga=\langle S \mid R\rangle$ has strictly 
positive deficiency (see \cite[p.\ 419]{Rob}). Free groups
provide examples of such groups.

Consider for illustration the examples of \cite[Sect.\ 3]{GL-loc}.
\begin{letlist}
\item
The \emph{hypercubic lattice} $\ZZ^n$ is the Cayley group of an abelian 
group with $|S|=2n$, $|R|=n+\binom n2$, and $\rank(C)=n$.
It has  a  \GHF\ (in fact, it has many, indexed by the non-zero, integer-valued
elements of $\sN(C)$).

\item The \emph{$3$-regular tree} is the Cayley graph of the
group with $S=\{s_1,s_2,t\}$, $R=\{s_1t,s_2^2\}$, and $\rank(C)=2$.
It has a \GHF.

\item The \emph{discrete Heisenberg group} has $|S|=|R|=6$
and $\rank(C)=4$. It has a \GHF.

\item
The \emph{square/octagon lattice} is
the Cayley graph of a finitely presented group with $|S|=3$ and
$|R|=5$, and this
does not satisfy the hypothesis of Theorem \ref{thm3}(b) (since $\rank(C)=3$). 
This presentation has no \GHF. Neither
does the lattice have a graph height function with automorphism 
subgroup acting transitively, but nevertheless
it possesses a \ughf\ in the sense of 
Definition \ref{def:height}, as explained in 
\cite[Sect.\ 3]{GL-loc}.

\item The \emph{hexagonal lattice} is the Cayley graph of
the finitely presented group with $S=\{s_1,s_2,s_3\}$ and
$R=\{s_1^2,s_2s_3, s_1s_2^2s_1s_3^2\}$. Thus, $|R|=|S|=3$,
$\rank(C)=2$, and the graph has a \GHF.

\end{letlist}

A discussion is presented in Section \ref{sce} of certain types of infinite groups
whose Cayley graphs have group or graph height functions. We describe next
some illustrative examples and a question.
The next proposition is extended in Theorem \ref{iq}.

\begin{proposition}\label{prop:abel}
Any finitely generated group which is infinite and abelian has a 
\GHF\ $h$ with $d(h)=1$.
\end{proposition}

\begin{example}\label{ex:ladder}
The \emph{infinite dihedral group} 
$\dih_\oo=\langle s_1,s_2\mid  s_1^2,s_2^2 \rangle$ is an example of an infinite, 
finitely generated group $\Ga$ which has no \GHF.
The Cayley graph of $\Ga$  is the
line $\ZZ$ with (harmonic) \ughf\ $(h,\sH)$, where $h$ is
the identity and $\sH$ is the group of shifts. 
This example of a solvable group is extended in
Theorem \ref{thm:vs}.
\end{example}

\begin{example}\label{ex:higman}
The \emph{Higman group}  $\Ga$  of \cite{gh} is an infinite, finitely presented group with presentation $\Ga=\langle S \mid R\rangle$ where
\begin{align*}
S&=\{ a,b,c,d,a',b',c',d'\},\\ 
R&=\{aa',bb',cc',dd'\}\cup \{a'ba(b')^2,b'cb(c')^2,c'dc(d')^2,
d'ad(a')^2\}.
\end{align*}
The quotient of $\Ga$ by its maximal proper normal 
subgroup is an infinite, finitely generated, simple group. 
By Theorem \ref{thm3}(b), $\Ga$ has no \GHF.
\end{example}

\begin{remark}\label{rem:grig0}
Since writing this paper, the authors have shown 
in \cite{GL-amen} that the Cayley graph of  the Higman group 
does not possess a \ghf.  
\end{remark}

\begin{proof}[Proof of Theorem \ref{thm3}]
(a) Let $\g$ be as given. To check that $h$ is well defined by \eqref{eq:gheight},
we must show that $h(v)$ is independent of the chosen representation
of $v$ as a word. Suppose that
$v=s_1\cdots s_m=u_1\cdots u_n$ with $s_i, u_j \in S$,
and extend the definition of $\g$ to the directed edge-set of $G$ by
\begin{equation}\label{eq:ga}
\g([g,gs\rangle) = \g_s, \qq g \in \Ga,\ s \in S.
\end{equation}
The
walk $(\id,s_1,s_1s_2,\dots,v)$ is denoted as $\pi_1$, and
$(\id,u_1,u_1u_2,\dots,v)$ as $\pi_2$, and the
latter's reversed walk as $\pi_2^{-1}$.
Consider the walk $\nu$ obtained by
following $\pi_1$, followed by $\pi_2^{-1}$.
Thus $\nu$ is a closed walk of $G$ from $\id$.

Any $\rho\in R$ gives rise to a directed cycle in $G$ through
$\id$, and we write $\Ga R$ for the set of images of such
cycles under the action of $\Ga$.
Any closed walk lies in the vector space over $\ZZ$ generated by
the directed cycles of  $\Ga R$ (see, for example, \cite[Sect.\ 4.1]{Ham}).
The sum of the $\g_s$ around any $g \rho\in \Ga R$ is zero, by 
\eqref{eq:ga} and the fact that $C\g=\bzero$. Hence
\begin{equation*}
\sum_{i=1}^m \g_{s_i}-  \sum_{j=1}^n \g_{u_j}=0,
\end{equation*}
as required.

We check next that $(h,\Ga)$ is a \ghf. Certainly, $h(\id)=0$.
For $u,v\in V$, write $v=ux$ where $x=u^{-1}v$,
so that $h(v)-h(u) = h(x)$ by \eqref{eq:gheight}.
For $g\in \Ga$, we have that $gv = (gu)x$, whence
\begin{equation}\label{eq:extra1}
h(gv)-h(gu) = h(x) = h(v)-h(u).
\end{equation}
Since $\g \ne \bzero$, there exists $s\in S$ with $\g_s>0$.
For $v \in V$, we have $h(vs^{-1}) < h(v) < h(vs)$. 

(b) The null space $\sN(C)$ is non-trivial if and only if $\rank(C)<|S|$.
Since $C$ has integer entries and $|S|<\oo$, $\sN(C)$ is non-trivial if and only if it contains a non-zero 
vector of integers. 

(c) See Remark \ref{rem:betti}. This may also be proved directly, but we omit
the details.

(d) 
We do not give the details of this, since a more general fact is proved in 
Proposition \ref{prop11}(b). The current proof follows that of the latter proposition
with $\sH=\Ga$, $F=h$, and $\Ga$ acting on $V$ by left-multiplication. Since 
this action of $\Ga$
has no non-trivial fixed points, $\Ga$ is unimodular.
\end{proof}

\begin{proof}[Proof of Proposition \ref{prop:abel}]
See Remark \ref{rem:betti} for an elementary group-theoretic explanation.
Since $\Ga$ is infinite and abelian, there exists a generator, $\si$ say,
of infinite order. For $s\in S$, let
\begin{equation}\label{eq:unit}
\g_s = \begin{cases} 1 &\text{if } s=\si,\\
-1 &\text{if } s=\si^{-1},\\
0 &\text{otherwise}.
\end{cases}
\end{equation}
Since any relator must contain equal numbers of appearances of $\si$ and $\si^{-1}$,
we have that $\g \in \sN(C)$. Therefore, the function $h$ of \eqref{eq:gheight}
is a \GHF. 
\end{proof}

\section{Cayley graphs with height functions}\label{sce}

The main result of this section is as follows. The associated definitions are presented later,
and the proofs of the next two theorems are at the end of this section.

\begin{theorem}\label{thm:vs}
Let $\Ga$ be an infinite, finitely generated group with a normal 
subgroup $\Ga^*$ satisfying $[\Ga:\Ga^*]<\oo$. 
Let $q=\sup\{i: [\Ga^*: \Ga^*_{(i)}]<\oo\}$ where $(\Ga^*_{(i)} : i \ge 1)$ 
is the derived series of $\Ga^*$.
If $q<\oo$ and $[\Ga^*_{(q)},\Ga^*_{(q+1)}]=\oo$, 
then every Cayley graph of $\Ga$ has a \ughf\ 
of the form $(h,\Ga^*_{(q)})$ which is harmonic.
\end{theorem}

The theorem may be applied to any finitely generated, 
virtually solvable group $\Ga$, and more generally whenever the derived series
of $\Ga^*$ terminates after finitely many steps at a \emph{finite} perfect group.

In preparation for the proof, we present  a general
construction of a height function
for a group having  a normal subgroup. 
Part (a) extends Proposition \ref{prop:abel} (see also Remark \ref{rem:betti}). 
 
\begin{theorem}\label{iq}
Let $\Ga$ be an infinite, finitely generated group, and let $\Ga'\normal \Ga$.
\begin{letlist}
\item If the quotient group $\Ga/\Ga'$ is infinite and abelian, 
then  $\Ga$ has a \GHF\ $h$ with $d(h)=1$.
\item If the quotient group $\Ga/\Ga'$ is finite, and $\Ga'$ has
a \GHF, then every Cayley graph of $\Ga$ has a harmonic, \ughf\ of the form $(h, \Ga')$. 
\end{letlist}
\end{theorem}

Recall that $\Ga/\Ga'$ is abelian if and only if $\Ga'$ contains the commutator
group $[\Ga,\Ga]$, of which the definition follows. 
An example of  Theorem \ref{iq}(b) in action is the special linear group 
$\SL_2(\ZZ)$ of the forthcoming Example \ref{ex:sl2z} (see \cite[p.\ 66]{dlH}).

We turn now towards solvable groups.
Let $\Ga$ be a group with identity $\id_\Ga$. 
The \emph{commutator} of the pair  $x,y\in \Ga$ is the group element
$[x,y]:=x^{-1}y^{-1}xy$.
Let $A$, $B$ be subgroups of $\Ga$. The \emph{commutator subgroup} $[A,B]$ is defined to be
\begin{equation*}
[A,B]=\bigl\langle [a,b] : a\in A,\ b\in B\bigr\rangle,
\end{equation*}
that is, the subgroup generated by all commutators $[a,b]$ with $a\in A$, $b\in B$. 
The \emph{commutator subgroup} of $\Ga$ is the subgroup $[\Ga,\Ga]$.
It is standard that $[\Ga,\Ga]\normal\Ga$, and the quotient group
$\Ga/[\Ga,\Ga]$ is abelian. The group $\Ga$ is called \emph{perfect}
if $\Ga=[\Ga,\Ga]$.

Here is an explanation of the terms in Theorem \ref{thm:vs}.
Let $\Ga_{(1)}=\Ga$. The \emph{derived series} of $\Ga$ is given recursively by the formula
\begin{equation}\label{eq:deriv}
\Ga_{(i+1)}=[\Ga_{(i)},\Ga_{(i)}], \qq i \ge 1.
\end{equation}
The group $\Ga$ is called \emph{solvable} if there exists an integer $c\in\NN$ 
such that $\Ga_{(c+1)}=\{\id_\Ga\}$. Thus, $\Ga$ is solvable if
there exists $c \in \NN$ such that
\begin{equation*}
\Ga=\Ga_{(1)}\normalgr \Ga_{(2)}\normalgr \cdots \normalgr \Ga_{(c+1)}=\{\id_{\Ga}\}.
\end{equation*}
A \emph{virtually solvable group} is a group $\Ga$ for which there exists a normal subgroup
$\Ga^*$  which is solvable and satisfies $[\Ga: \Ga^*] < \oo$.
The reader is referred to \cite{LedW,Rob} for general accounts of group theory. 

\begin{example}\label{ex:lampl}
Here is an example of a finitely generated but not finitely presented group with a \GHF.
The \emph{lamplighter group} $L$ has presentation
$\langle S\mid R\rangle$ where $S=\{a,t,u\}$ and $R=
\{a^2, tu\}\cup\{[a,t^nau^{n}]: n\in \ZZ\}$.
It has a \GHF\ since the rank of its coefficient matrix is $2$.
A recent reference to linear-growth harmonic functions on $L$ is \cite{BDCKY}.
\end{example}

\begin{example}\label{ex:sl2z}
The special linear group $\Ga:=\SL_2(\ZZ)$ has a presentation
\begin{equation} \label{sl2z}
\Ga=\langle x,y,u,v\mid xu, yv, x^4, x^2v^{3} \rangle,
\end{equation}
where
\begin{equation*}
x=\begin{pmatrix}0&-1\\1 &0\end{pmatrix},\qq
 y=\begin{pmatrix}0&-1\\1&1\end{pmatrix}.
\end{equation*}
The presentation has no \GHF.

The commutator subgroup $\Ga^{(2)}:=[\Ga,\Ga]$ is a normal 
subgroup of $\Ga$ with index $12$, and $\Ga^{(2)}$ is free of rank $2$. 
(See  \cite{Kon}
and \cite[p.\ 66]{dlH}.) By Theorem \ref{iq}(b),
every Cayley group of $\Ga$ has a harmonic, \ughf.
\end{example}

\begin{proof}[Proof of Theorem \ref{iq}]
(a) This is an immediate consequence of Remark \ref{rem:betti}. A detailed argument
may be outlined as follows. Let $\Ga=\langle S\mid R\rangle$. 
If $Q:=\Ga/\Ga'$ is infinite and abelian, it is generated
 by the cosets $\{\ol s:=s\Ga': s \in S\}$, 
and its relators are the
words $\ol s_1\ol s_2\cdots \ol s_r$ as  $\rho=s_1s_2\cdots s_r$ ranges
over $R$.  Choose $\si\in S$ with infinite order, and let 
\begin{equation}\label{eq:unit2}
\g_s=\begin{cases} 1 &\text{if } s\in \ol\si,\\
-1 &\text{if } s^{-1} \in  \ol\si,\\
0 &\text{otherwise}.
\end{cases}
\end{equation}
It may now be checked that $C\g=\bzero$ where $C$ is the coefficient 
matrix. 

\noindent
(b) 
Let $G$ be a Cayley graph of $\Ga$, and let $\Ga'\normal \Ga$
satisfy $[\Ga:\Ga']<\oo$. 
By assumption, $\Ga'$ has a \GHF\ $h'$. 
The subgroup $\Ga'$ of $\Ga$ acts on $G$ by left-multiplication, and it
is unimodular since its elements act with no non-trivial fixed points.
We apply Theorem \ref{prop:indep} with $\sH=\Ga'$ and $F=h'$ 
to obtain a harmonic, \ughf\ $(h,\Ga')$ on $G$.
\end{proof}

\begin{proof}[Proof of Theorem \ref{thm:vs}]
Since $q<\oo$, we have that $[\Ga:\Ga^*_{(q)}]<\oo$, and in particular
$\Ga^*_{(q)}$ is finitely generated.
Now, $\Ga^*_{(q)}$ is characteristic in $\Ga^*$, and $\Ga^*\normal \Ga$, so that
$\Ga^*_{(q)} \normal \Ga$.

By applying Theorem \ref{iq}(a) to the pair $\Ga_{(q+1)}^* \normal \Ga^*_{(q)}$, 
there exists a \GHF\ $h^*_q$
on $\Ga^*_{(q)}$. 
We apply Theorem \ref{iq}(b) to the pair $\Ga^*_{(q)}\normal\Ga$ to obtain 
a harmonic, \ughf\ $(h, \Ga^*_{(q)})$ on $\Ga$.
\end{proof}

\section{Convergence of connective constants of Cayley graphs}\label{sec:conv}

Let $\Ga=\langle S \mid R\rangle$ be finitely presented with coefficient matrix $C$
and Cayley graph $G=G(S,R)$.
Let $t \in \Ga$ have infinite order. We consider in this section the effect of adding a new relator
$t^m$, in the limit as $m \to\oo$. Let $G_m$ be the Cayley graph of
the group $\Ga_m=\langle S\mid R\cup\{t^m\}\rangle$.

\begin{theorem}\label{torus}
If $\rank(C) < |S|-1$, then 
 $\mu(G_m) \to \mu(G)$ as $m\to\oo$.
\end{theorem}

\begin{proof}
The coefficient matrix  $C_m$ of $G_m$ differs from $C_1$ only in  the multiplicity
of the row corresponding to the new relator, and therefore $\sN(C_1)=\sN(C_m)$.
Since $\Ga_1$ has only one relator more than $G$, $\rank(C_1) \le \rank(C)+1$.
If $\rank(C)<|S|-1$, then
$\rank(C_1)<|S|$. By Theorem \ref{thm3}, we may find $\g=(\g_s:s\in S)\in\sN(C_1)$ such that
$\g\in\ZZ^S$, $\g\ne\bzero$. By the above, for $m \ge 1$, $\g\in\sN(C_m)$, 
so that  $G_m$ has a corresponding 
\GHF\ $h_m$. By \eqref{def:d},
$d(h)=d(h_m) = :D$ for all $n$,
so that $G_m \in \sG_{D,0}$ for all $m$.

The group $\Ga_m$ is obtained as the quotient group
of $\Ga$ by the (normal) subgroup generated by $t^m$.
We apply \cite[Thm 5.2]{GL-loc} with $\sA_m$ the cyclic group generated by $t^m$. 
The condition of the theorem holds since $t$ has infinite order.
\end{proof}

\begin{remark}[Approximating $\mu(G)$]\label{rem:app}
The question is posed in \cite{GrL3} of whether one can obtain rigorous sequences
of bounds for $\mu(G)$ which are sharp. Such upper bounds are provided
by the subadditive argument of \eqref{eq:sisub}, namely $\mu\le \si_m^{1/m}$ for $m \ge 1$.
Theorem \ref{torus}, taken together with \cite[Thm 3.8]{GrL3}, provides lower bounds.
It is however preferable to use the improved lower bound $\mu \ge b_m^{1/m}$,
where $b_m$ is the number of $m$-step bridges. 
The latter inequality is asymptotically sharp whenever $G$ has a \ughf\ (see \cite[Remark 4.5]{GL-loc}),
and this is a less restrictive condition than that of Theorem \ref{torus}.
\end{remark}

As examples of finitely generated
groups satisfying the conditions of Theorem \ref{torus}, we mention
free groups, abelian groups, free nilpotent groups, free solvable groups, and, more widely,
nilpotent and solvable groups $\Ga$
with presentations $\langle S\mid R\rangle$
whose coefficient matrix $C$ satisfies
$b(\Ga)=|S|-\rank(C) >1$. Here is an
example where Theorem \ref{torus} cannot be applied,
though the conclusion is valid.

\begin{example}\label{ex:dih2}
Let $G$ be the Cayley graph of the \emph{infinite dihedral group}
$\dih_\oo = \langle s_1,s_2 \mid s_1^2,s_2^2\rangle$ of Example \ref{ex:ladder}.
As noted there, $G$ has no \GHF, though it has a \ughf\ $(h,\sH)$ with $d(h)=1$.
Let $\Ga_m=\dih_\oo \times J_m$ where 
$m \ge 3$ and $J_m=\langle a,b\mid ab, a^m\rangle$
is the cyclic group $\{\id,a,a^2,\dots, a^{m-1}\}$. Thus, $\Ga_m$ is finitely presented
but, by Theorem \ref{thm3}(b), it
has no \GHF. In particular, Theorem \ref{torus} may not be applied.

The Cayley graph $G$  is isomorphic to $\ZZ$.
Therefore, we may define a \ughf\ $(h',\sH')$ on the Cayley graph $G_m$ 
of $\Ga_m$ by
$h'(\g, a^k)=h(\g)$, with $\sH'$ generated by the shifts $(\g,a^k)\mapsto (\g+1,a^k)$ and
$(\g,a^k)\mapsto (\g,a^{k+1})$. Furthermore,
$d(h') = d(h)=1$ and $r(h',\sH')=r(h,\sH)=0$.
By \cite[Thms 5.1, 5.2]{GL-loc}, $\mu(G_m)\to\mu(\ZZ^2)$ as $m\to\oo$.
\end{example}

\section{Proof of Theorem \ref{prop:indep}}\label{sec:pf-}

Assume that assumptions (a)--(c) of Theorem \ref{prop:indep} hold. 
There are two steps in the proof,
namely of the following.
\begin{Alist}
\item (Prop.\ \ref{102})
There exists $\psi:V \to \QQ$ which is \hdi, harmonic,
non-constant, and takes values in the rationals.
\item (Prop.\ \ref{prop12})
There exists a \ghf\ which is harmonic on $G$.
\end{Alist}

The vertex $\id$ may appear to play a distinguished role in this section.
This is in fact not so: since $G$ is assumed transitive, the following is valid with any choice 
of vertex for the label $\id$. 
The approach of the proof is inspired in part by the proof of
\cite[Cor.\ 3.4]{LeeP}.
Let $X=(X_n: n=0,1,2,\dots)$ be a simple random walk on $G$, with transition matrix 
\begin{equation*} 
P(u,v)=\PP_u(X_1=v)=\frac{1}{\deg(u)}, \qq u,v\in V,\  v\in\pd u,
\end{equation*}
where $\PP_u$ denotes the law of the random walk starting at $u$.

Let $V_1=\sH\id$ be the orbit of the identity under $\sH$, and
let $P_1$ be the transition matrix of the induced random walk on $V_1$, that is 
\begin{equation*}
P_1(u,v)=\PP_u(X_{\tau}=v), \qq u,v \in V_1,
\end{equation*}
where $\tau=\min\{n\geq 1:X_n\in V_1\}$. It is easily seen that
$\PP_u(\tau<\oo)=1$ since, by the quasi-transitive action of $\sH$,
there exist $\a>0$ and $K<\oo$ such that
\begin{equation}\label{eq:bnd}
\PP_u(X_k\in V_1\text{ for some $1\le k \le K$}) \ge \a, \qq u\in V.
\end{equation}
We note for later use that, by \eqref{eq:bnd}, there exist $\a'=\a'(\a,K)\in(0,1)$ and $A=A(\a,K)$ such that
\begin{equation}\label{eq:bnd2}
\PP_u(\tau\ge m) \le A(1-\a')^m, \qq m \ge 1, \ u\in V.
\end{equation}

Since $\sH\le \Aut(G)$, $P_1$ is invariant under $\sH$ in the sense that
\begin{equation}\label{eq:inv}
P_1(u,v) = P_1(\g u, \g v), \qq \g \in \sH,\ u,v\in V_1.
\end{equation}

\begin{proposition}\label{prop11}
\mbox{}
\begin{letlist}
\item 
The transition matrix $P_1$ is symmetric, in that 
$$
P_1(u,v)=P_1(v,u), \qq u,v\in V_1.
$$
\item Let $F:V_1\to \ZZ$ be \hdi. 
Then $F$ is $P_1$-harmonic in that
$$
F(u) = \sum_{v\in V_1} P_1(u,v)F(v), \qq u \in V_1.
$$
\end{letlist} 
\end{proposition}

\begin{proof}
(a) 
Since $P$ is reversible with respect to the measure $(\deg(v): v \in V)$,
and $\deg(v)$ is constant on $V_1$, we have that
$$
P(u_0,u_1)P(u_1,u_2)\cdots P(u_{n-1},u_n) = 
P(u_n,u_{n-1})P(u_{n-1},u_{n-2})\cdots P(u_1,u_0)
$$ 
for $u_0,u_n \in V_1$, $u_1,\dots,u_{n-1} \in V$. The symmetry of $P_1$ follows
by summing over appropriate sequences $(u_i)$. 

\noindent
(b) It is required to prove that
\begin{equation}\label{poh}
\sum_{v\in V_1}P_1(u,v)[F(u)-F(v)]=0, \qq u\in V_1,
\end{equation}
and it is here that we shall use assumption (b) of Theorem \ref{prop:indep},
namely, that $\sH$ is unimodular. 
Since $F$ is \hdi, there exists $D<\oo$ such that
$$
|F(u)-F(v)| \le D d_G(u,v), \qq u,v\in V_1.
$$
By \eqref{eq:bnd},
the random walk on $V_1$ has finite mean step-size.
It follows that the summation in \eqref{poh} converges absolutely.

Equation \eqref{poh} may be proved by a cancellation of summands, but it is shorter to
use the mass-transport principle.  Let
\begin{equation*}
m(u,v) = P_1(u,v)[F(u)-F(v)],\qq u,v\in V_1.
\end{equation*}
The sum $\sum_{v\in V_1}m(u,v)$ is absolutely convergent as above, and
$m(\g u,\g v)= m(u,v)$ for $\g \in \sH$. Since $\sH$ is unimodular,
by the mass-transport principle (see, for example, \cite[Thm 8.7, Cor.\ 8.11]{LyP}),
\begin{equation}\label{eq:mtp1}
\sum_{v\in V_1} m(u,v) = \sum_{w\in V_1} m(w,u), \qq u \in V_1.
\end{equation}
Now,
\begin{align*}
\sum_{w\in V_1} m(w,u) &= \sum_{w\in V_1} P_1(w,u)[F(w)-F(u)]\\
&= -\sum_{w\in V_1} P_1(u,w)[F(u)-F(w)] \qq\text{by part (a)},
\end{align*}
and \eqref{poh} follows by \eqref{eq:mtp1}.

It is usual to assume in the mass-transport principle that $m(u,v) \ge 0$, but
it suffices that $\sum_v m(u,v)$ is absolutely convergent.
\end{proof}

Let $\be>1$, and let $f:V\to\RR$. We write $f=\O(\beta^n)$ if there exists $B$ such that
\begin{equation}\label{eq:expgrowth}
|f(v)|\le B\beta^n\qq \text{if } d_G(\id,v)\le n, \text{ and } n \ge 1.
\end{equation}

\begin{proposition}\label{102}
Let $F:V_1 \to \ZZ$ be \hdi, and 
let
\begin{equation}\label{eq:psidef}
\psi(v)=\EE_v[F(X_T)],\qq v \in V,
\end{equation}
where $T=\inf\{n\ge 0: X_n\in V_1\}$. Then,
\begin{letlist}
\item the function $\psi$ is \hdi,  and agrees with $F$ on $V_1$,
\item  $\psi$ is harmonic on $G$, in that
\begin{equation}\label{eq:harm2}
\psi(u) = \sum_{v\in V} P(u,v)\psi(v), \qq u \in V,
\end{equation}
and, furthermore, $\psi$ is the unique harmonic function
that agrees with $F$ on $V_1$ and satisfies $\psi=\O(\beta^n)$
with any $1\le \beta< 1/(1-\a')$, where $\a'$ satisfies \eqref{eq:bnd2},
\item $\psi$ takes rational values.
\end{letlist}
\end{proposition}

\begin{remark}\label{rem:unique}
By Proposition \ref{102}(a,\,b), any $\O(\be^n)$ harmonic extension of $F$ (with suitable $\be$) 
is \hdi. Conversely, any \hdi\ function $f$ satisfies $f=\O(\be^n)$
for all $\beta>1$, whence the function $\psi$ 
of \eqref{eq:psidef} is the unique harmonic extension of $F$ that
is \hdi.
\end{remark}

\begin{proof}
(a) The function $\psi$ is \hdi\ since the law of   
the random walk is $\sH$-invariant, and
$$
\psi(v)-\psi(w) = \EE_v[F(X_T)] -\EE_w[F(X_T)].
$$
It is trivial that $\psi\equiv F$ on $V_1$.

\noindent
(b) By conditioning on the first step, 
$\psi$ is harmonic at any $v \notin V_1$.
For $v\in V_1$, it suffices to show that
$$
\psi(v) = \sum_{w\in V}P(v,w)\psi(w).
$$
Since $\psi\equiv F$ on $V_1$, and $F$ is $P_1$-harmonic
(by Proposition \ref{prop11}), this may be written as
\begin{equation*}
\sum_{w\in V_1}P_1(v,w)\psi(w) = \sum_{w\in V} P(v,w)\psi(w), \qq v \in V_1.
\end{equation*}
Each term equals $\EE_v[\psi(W(X_1))]$, where $X_1$ is the
position of the random walk after one step,  and $W(X_1)$ is
the first element of $V_1$ encountered having started at $X_1$.

To establish uniqueness, let $\phi$ be a harmonic function with 
$\phi=\O(\beta^n)$ where
$1\le\beta < 1/(1-\a')$, such that $\phi\equiv F$ on $V_1$.
Then $Y_n:= \phi(X_n)$ is a martingale, and furthermore
$T$ is a stopping time with tail satisfying \eqref{eq:bnd2}.
By the optional stopping theorem (see, for example, \cite[Thm 12.5.1]{GS})
and \eqref{eq:psidef}, 
$$
\phi(u)=\EE_u(Y_T) = \EE_u(F(X_T)) = \psi(u),
$$
so long as $\EE_u(|Y_n|I_{\{T\ge n\}})\to 0$ as $n \to\oo$,
where $I_E$ denotes the indicator function of an event $E$.
To check the last condition, note by \eqref{eq:expgrowth} and \eqref{eq:bnd2}  that
\begin{align*}
\EE_u(|Y_n|I_{\{T\ge n\}}) &\le B\beta^{n+|u|}\PP_u(T\ge n)\\
&\le (AB\beta^{|u|})\beta^n(1-\a')^n\to 0 \qq\text{as } n\to\oo,
\end{align*}
where $|u|=d_G(\id,u)$.

\noindent
(c) The quantity $\psi(v)$  
has a representation as a sum of values of the unique 
solution of a finite set of linear equations with integral coefficients and boundary
conditions, and thus $\psi(v)\in\QQ$. Some further details follow.

Let  $\vG=(V, \vE)$ be the directed graph obtained from $G=(V,E)$ by
replacing each $e\in E$ by two edges $\vec e$, $-\vec e$
with the same endpoints and opposite orientations.
Suppose $\de: \vE\to \RR$ satisfies the linear equations
\begin{alignat}{2}
\de(-\vec e) +\de(\vec e)&=0, \qq &&e \in E,\label{eq:negsym}\\
\sum_{\vec e \in W} \de(\vec e)&=0, \qq && W \in \sW(G),\label{eq:cycles2}\\
\sum_{v \sim  u}\de([ u, v\rangle) &= 0,\qq  &&u \in V,\label{eq:harm}\\
\de(\a \vec e) &= \de(\vec e),\q &&e \in E,\ \a \in \sH,\label{eq:alv}
\end{alignat}
where $\sW(G)$ is the set of directed closed walks of $G$. 
(Equation \eqref{eq:cycles2} may be viewed as including \eqref{eq:negsym}.)
By \eqref{eq:cycles2}, the sum
$$
\De(v) := \sum_{\vec e \in \ell_v} \de(\vec e),\qq v \in V,
$$
is well defined, where $\ell_v$ is an arbitrary (directed) walk from $\id$ to $v\in V$.
Equation \eqref{eq:harm} requires that $\De$ be harmonic, and \eqref{eq:alv}
that $\De$ be \hdi. 

Since $\sH$ acts quasi-transitively, by \eqref{eq:alv},
the linear equations \eqref{eq:negsym}--\eqref{eq:harm} involve only
 finitely  many variables.  Therefore, there
exists a finite subset of equations, denoted as $\Eq$, of \eqref{eq:negsym}--\eqref{eq:harm} 
such that $\de$ satisfies \eqref{eq:negsym}--\eqref{eq:alv} if
only if $\de$ satisfies $\Eq$ together with \eqref{eq:alv}.
In summary, any harmonic, \hdi\ function $\De$, satisfying $\De(\id)=0$,  corresponds to
a solution to the finite collection $\Eq$ of linear equations.

With $F$ as given, let $\psi$ be given by \eqref{eq:psidef}. 
By Remark \ref{rem:unique},
equations \eqref{eq:negsym}--\eqref{eq:alv} have a
unique solution  satisfying
\begin{equation}\label{eq:ne}
\sum_{\vec e \in \ell_v} \de(\vec e) =F(v)-F(\id), \qq v \in V_1.
\end{equation}
By \eqref{eq:cycles2}, it suffices in \eqref{eq:ne} to 
consider only the finite set $V_1'\subseteq V_1$ of vertices $v$ within some bounded distance of $\id$
that depends on the pair $G$, $\sH$.

Therefore, $\Eq$ possesses a unique solution subject to \eqref{eq:ne}
(with $V_1$ replaced by $V_1'$). All coefficients
and boundary values  in $\Eq$ and \eqref{eq:ne} are integral, 
and therefore $\psi$ takes only rational values.
\end{proof}

\begin{proposition}\label{prop12}
Let  $F:V_1 \to \ZZ$ be \hdi, and non-constant on $V_1$.
There exists a \ghf\  
$(h=h_F,\sH)$ such that $h$ is harmonic on $G$.
\end{proposition}

\begin{proof}
The normality of $\sH$ is used in this proof.
A vertex $v\in V$ is called a \emph{\sap} of a function $h:V \to\RR$ if $v$ has
neighbours $u$, $w$ such that $h(u)<h(v)<h(w)$. The function $h$ is said to
\emph{increase everywhere} if every vertex is a point of increase.
For $v\in V$ and a harmonic function
$h$, 
\begin{equation}\label{eq:fork}
\begin{aligned}
\mbox{either: }&\mbox{ $v$ is a \sap\ of $h$,}\\
\mbox{or: }&\mbox{ $h$ is constant on $\{v\}\cup \pd v$.}
\end{aligned} 
\end{equation}
An \hdi\ function $h$ on $G$ is a \ghf\ if and only if $h(\id)=0$, $h$ takes
integer values, and $h$ increases everywhere. 

Let $F$ be as given, and let $\psi$ be given by Proposition \ref{102}.
Thus, $\psi:V \to\QQ$ is non-constant on $V_1$, \hdi, and harmonic on $G$. 
Since $\psi$ is \hdi, we may replace it by $m\psi$ for a suitable $m\in\NN$
to obtain such a function that in addition takes integer values.
We shall work with the latter function, and thus we assume henceforth that $\psi:V \to\ZZ$.
Now, $\psi$ may not increase everywhere.
By \eqref{eq:fork}, $\psi$ has some \sap\  $w \in V$. 

Let $V_1,V_2,\dots,V_M$ be the orbits of $V$ under $\sH$. Find $W$ such that
$w \in V_W$.
Since $\Ga$ acts transitively on $G$, and $\sH$ is a normal subgroup of $\Ga$ acting 
quasi-transitively on $G$, there exist $\g_1,\g_2,\dots,\g_M\in \Ga$ such that
$\g_W=\id_\Ga$ and
\begin{equation*}
V_i=\g_iV_W,\qq  i=1,2,\dots,M.
\end{equation*} 
Let
\begin{equation}\label{eq:psi2def}
\psi_i(v)=\psi(\g_i^{-1}v), \qq i=1,2,\dots,M,
\end{equation}
so that, in particular, $\psi_W=\psi$.
Since $w \in V_W$ is a \sap\ of $\psi$, $w_i := \g_iw$ is a \sap\ 
of $\psi_i$, and also $w_i \in V_i$.  

\begin{lemma}\label{lem3} 
For $i=1,2,\dots,M$,
\begin{letlist}
\item $\psi_i:V \to\ZZ$ is a non-constant, harmonic function, and
\item $\psi_i$ is \hdi.
\end{letlist}
\end{lemma}

\begin{proof}
(a) Since $\psi_i$ is obtained from $\psi$ by shifting the domain 
according to the automorphism $\g_i$,
$\psi_i$ is non-constant and harmonic.

\noindent
(b) For $\a\in \sH$ and $u,v\in V$, 
\begin{equation*}
\psi_i(\a v)-\psi_i(\a u)=\psi(\g_i^{-1}\a v)-\psi(\g_i^{-1}\a u).
\end{equation*}
Since $\sH \normal \Ga$ and $\g_i\in\Ga$,
there exists $\a_i\in\sH$ such that $\g_i^{-1}\a=\a_i\g_i^{-1}$.
Therefore,
\begin{alignat*}{2}
\psi_i(\a v)-\psi_i(\a u)&=\psi(\a_i\g_i^{-1}v)-\psi(\a_i\g_i^{-1}u)\ &&\\
&=\psi(\g_i^{-1}v)-\psi(\g_i^{-1}u) &&\text{since $\psi$ is \hdi}\\
&=\psi_i(v)-\psi_i(w) &&\text{by \eqref{eq:psi2def},}
\end{alignat*}
so that $\psi_i$ is \hdi.
\end{proof}

Let $\nu:V \to \RR$ be \hdi. For $j=1,2,\dots, M$, either every vertex in $V_j$ is a \sap\ of $\nu$, 
or no vertex in $V_j$ is a \sap\ of $\nu$. We shall now
use an iterative construction in order to find a harmonic, \hdi\ function $h'$
for which every $w_i$ is a \sap.  Since the $w_i$ represent the orbits $V_i$, 
the ensuing $h'$ increases everywhere.

\begin{numlist}
\item If every $w_i$ is a \sap\ of $\psi$, we set $h'=\psi$.
\item Assume otherwise, and find the smallest $j_2$ such that
$w_{j_2}$ is not a \sap\ of $\psi$. 
By \eqref{eq:fork}, we may choose $c_{j_2}\in \QQ$ such that both $w$ and $w_{j_2}$ are
\saps\ of $h_2 := \psi+c_{j_2}\psi_{j_2}$. If $h_2$ increases everywhere,
we set $h' = h_2$. 

\item Assume otherwise, and find the smallest $j_3$ such that
$w_{j_3}$ is not a \sap\ of $h_2$. 
By \eqref{eq:fork}, we may choose $c_{j_3}\in \QQ$ such that $w$, $w_{j_2}$, and $w_{j_3}$ are
\saps\ of $h_3 := \psi+c_{j_2}\psi_{j_2}+c_{j_3}\psi_{j_3}$. If $h_3$ increases everywhere,
we set $h' = h_3$. 

\item This process is iterated until we find 
an \hdi, harmonic function  of the form 
$$
h' =\sum_{l=1}^M c_{j_l}\psi_{j_l},
$$
with $j_1=W$, $c_W=1$, and $c_{j_l} \in \QQ$, which increases everywhere.
\end{numlist}

The function $h'-h'(\id)$ may fail to be a \ghf\ only in that
it may take rational rather than integer values. 
Since the $c_{j_l}$ are rational, there exists $m\in \ZZ$ such that 
$h=m(h'-h(\id))$ is a  \ghf.
\end{proof}

\begin{proof}[Proof of Theorem \ref{prop:indep}]
By Propositions \ref{prop11} and \ref{102}, there exists $\psi:V\to\QQ$ satisfying (i).
The existence of $\psi':V\to\QQ$, in (ii), follows as in Proposition \ref{prop12}.
Similarly, $\psi$, $\psi'$ may be taken to be integer-valued, 
and the unimodularity holds since $\sH$ is assumed unimodular.
\end{proof}

\section{Proof of Theorem \ref{cor:nonunim}}\label{sec:nonunim}

Let $G$, $\Ga$, $\sH$ be as given. The idea is to apply Theorem \ref{prop:indep} 
to a suitable triple  $G'$, $\Ga'$, $\sH'$, and
to extend the resulting \ghf\ to the original graph $G$. The required function $F$ of the theorem
will be derived from the modular function of $G$ under $\sH$.

By \cite[Thm 8.10]{LyP}, we may define a positive \emph{weight function} $M:V\to(0,\oo)$ satisfying
\begin{equation}\label{eq:modf}
\frac{M(u)}{M(v)} = \frac{|\Stab_u v|}{|\Stab_v u|}, \qq u,v\in V,
\end{equation}
where $|\cdot|$ denotes cardinality. The weight function is uniquely
defined up to a multiplicative constant, and is automorphism-invariant
up to a multiplicative constant.
Since $G$ is assumed non-unimodular, $M$ is non-constant on some orbit
of $\sH$. Without loss of generality, we assume $\id$ lies in such an orbit and that $M(\id)=1$.
See \cite[Sect.\ 8.2]{LyP} for an account of (non-)unimodularity.

Let $\sS$ be the normal subgroup of $\Ga$ generated by $\bigcup_{v\in V} \Stab_v$,
where $\Stab_v=\Stab_v^\sH$.
Let $G'$ denote the quotient graph $G/\sS$ (as in \cite[Sect.\ 2]{GL-loc}), 
which we take to be simple in that
every pair of neighbours is connected by just one edge, and any loop is removed.

\begin{lemma}\label{lem:gprime}\mbox{}
\begin{letlist}
\item $\sS \normal \sH$.
\item The function $F':V/\sS\to(0,\oo)$ given by $F'(\sS v) = \log M(v)$, $v \in V$,
is well defined, in the sense that $F'$ is constant on each  orbit in $\sS$.
\item The quotient group $\Ga':=\Ga/\sS$ acts transitively on $G'$,
and $\sH':=\sH/\sS$ acts quasi-transitively on $G'$.
\item The quotient graph $G'=G/\sS$ satisfies $G' \in \sG$.
\item $\sH'$ is unimodular on $G'$.
\end{letlist}
\end{lemma}

\begin{proof}
(a)  Since $\sS\normal\Ga$ and $\sH \le \Ga$, it suffices to show that $\sS\le\sH$.
Now, $\sS$ is the set of all products of the form $(\g_1\si_1\g_1^{-1})(\g_2\si_2\g_2^{-1})
\cdots (\g_k\si_k\g_k^{-1})$ with $k \ge 0$, $\g_i\in\Ga$, $\si_i\in\Stab_{w_i}$, $w_i\in V$.
Since $\g_i\si_i\g_i^{-1} \in \Stab_{\g_iw_i}$, we have that  $\sS\le\sH$ as required.

\noindent
(b) If $u=\si v$ with $\si\in\Stab_w$, then
\begin{equation*}
\frac  {M(u)}{M(w)} = \frac{|\Stab_u w|}{|\Stab_w u|}
=\frac{|\Stab_{\si v}(\si w)|}{|\Stab_{\si w}(\si v)|}
= \frac{|\Stab_{ v} w|}{|\Stab_{ w} v|}=\frac{M(v)}{M(w)},
\end{equation*}
so that $M(u)=M(v)$. As in part (a), every element of $\sS$ is the product of members of 
the stabilizer groups $\Stab_w$, and the claim follows.

\noindent
(c) Let $u,v\in V$, and find $\g\in\Ga$ such that $v=\g u$. Since $\sS\normal \Ga$,
$\sS \g (\sS u)=\sS \g u=\sS v$, so that $\sS \g: \sS u\mapsto \sS v$.
The first claim follows, and the second is similar since $\sH$ acts quasi-transitively
on $G$.

\noindent
(d) Since $M$ is non-constant on the orbit $\sH\id$, there exist $v,w\in\sH\id$ such that
$\xi:=M(w)/M(v)$ satisfies $\mu>1$. Let $\a\in\sH$ be such that $w=\a v$. By \eqref{eq:modf},
$M(\a^k v)/M(v)=\xi^k$, whence the range of $M$ is unbounded. By part (b),
$G'$ is infinite. (The non-constantness of the modular function has been used also in \cite{GHP}.)
The graph $G'$ is connected since $G$ is connected, and  is transitive
by part (c). It is locally finite since its vertex-degree is no greater than that of $G$.

\noindent
(e) It suffices for the unimodularity that, for $u\in V$ and $\ol u:= \sS u$,  we have that 
$\Stab_{\ol u}:=\Stab^{\sH'}_{\ol u}$
is a single element, namely the identity element $\sS$ of $\sH'$.  
Let $\a\in\sH$ be such that $\sS \a \in \Stab_{\ol u}$.
Then $\sS\a (\sS u)=\a\sS u =\sS u$. Therefore, there exists $s \in \sS$ such that $\a s(u)=u$,
so that $\a s\in\sS$. It follows that $\a\in\sS$, and hence $\sS\a=\sS$ as required.
\end{proof}

Since $M$ is non-constant on $\sH\id$, $F'$ is non-constant on the
orbit of $\sH'$ containing $\sS\id$. By Theorem \ref{prop:indep} applied
to $(G',\Ga',\sH',F')$,
$G'$ has a harmonic, \ughf\ $(\psi',\sH')$ satisfying $\psi'(\sS\id)=0$. 
Let $\psi:V \to \ZZ$ be given by $\psi(v)=\psi'(\sS v)$. 
We claim that $(\psi,\sH)$ is a \ghf\ on $G$.

Firstly, for $\a\in\sH$,
\begin{alignat*}{2}
\psi(\a v)-\psi(\a u) &=\psi'(\sS\a v) - \psi'(\sS\a u)\\
&=\psi'(\a \sS v) - \psi'(\a \sS u) \q&&\text{since } \sS\normal \sH\\
&=\psi'(\sS v) - \psi'(\sS u) &&\text{since $(\psi',\sH')$ is a \ghf}\\
&=\psi(v)-\psi(u),
\end{alignat*}
whence $\psi$ is \hdi.
Secondly, let $v\in V$, and find $u, w\in \pd v$ such that
$\psi'(\sS u) < \psi'(\sS v) < \psi'(\sS w)$. Then
$\psi(u)<\psi(v)<\psi(w)$, so that $v$ is a point of
increase of $\psi$. Therefore, $(\psi,\sH)$ is a \ghf\ on $G$.

Finally, we give an example in which the above recipe leads to a \ghf\
which is not harmonic. Consider the \lq grandparent graph' introduced
in \cite{Trof} (see also \cite[Example 7.1]{LyP}) and defined as follows. 
Let $T$ be an infinite degree-$3$ tree, and
select an \lq end' $\om$. For each vertex $v$, we add an edge to the unique
grandparent of $v$ in the direction of $\om$. Let $\sH$ be the set of automorphisms of the
resulting graph $G$. Note that $\sH$ acts transitively on $G$,
and is non-unimodular. The above recipe yields (up to
a multiplicative constant which we take to be $1$) the \ghf\ on $T$ which
measures the (integer) height of a vertex in the direction of $\om$.
The neighbours of a vertex with height $h$ have average height
$h-\frac 78$, whence $h$ is not harmonic.  

\section*{Acknowledgements} 

This work was supported in part
by the Engineering and Physical Sciences Research Council under grant EP/I03372X/1. 
ZL's research is supported by the Simons Foundation $\#$351813 and the
National Science Foundation DMS-1608896.
Jack Button kindly pointed out an error in an earlier version of this work,
and gave advice on certain group-theoretic aspects. The authors are grateful to
Yuval Peres for his comments on the work, and for proposing
the non-unimodular case of Theorem \ref{cor:nonunim}. 

\providecommand{\bysame}{\leavevmode\hbox to3em{\hrulefill}\thinspace}
\providecommand{\MR}{\relax\ifhmode\unskip\space\fi MR }
\providecommand{\MRhref}[2]{%
  \href{http://www.ams.org/mathscinet-getitem?mr=#1}{#2}
}
\providecommand{\href}[2]{#2}

\end{document}